\documentclass[reqno]{amsart}

\usepackage{latexsym}
\usepackage{amsmath,amsfonts,amssymb,epsfig}
\usepackage{amsthm,amscd}
\usepackage{mathrsfs} 
\usepackage{hyperref}
\usepackage{diagrams}
\usepackage{wrapfig}
\usepackage{cite}

\usepackage{tikz}
\usetikzlibrary{arrows,chains,matrix,positioning,scopes}

\makeatletter
\tikzset{join/.code=\tikzset{after node path={%
\ifx\tikzchainprevious\pgfutil@empty\else(\tikzchainprevious)%
edge[every join]#1(\tikzchaincurrent)\fi}}}
\makeatother

\tikzset{>=stealth',every on chain/.append style={join},
         every join/.style={->}}

\newtheorem{thm}{Theorem}[section]
\newtheorem*{mthm}{Main Theorem}
\newtheorem{cor}[thm]{Corollary}
\newtheorem{lem}[thm]{Lemma}

\newtheorem{prop}[thm]{Proposition}
\newtheorem{rem}[thm]{Remark}
\newtheorem{prob}[thm]{Question}

\theoremstyle{definition}

\numberwithin{equation}{section}

\newcommand{\no}{\noindent}

\def\R{\mathbb{R}}
\def\Q{\mathbb{Q}}
\def\Z{\mathbb{Z}}
\def\T{\mathbb{T}}
\def\Cnf{\text{\rm Conf}} 

\graphicspath{{./pict/},{./../pict/},{./../../pict/}}

\begin{document}

\title{Homotopy Brunnian links and the $\kappa$-invariant}


\author{F. R. Cohen} 

\address{
University of Rochester,
Rochester, New York 14627 } 
\email{cohf@math.rochester.edu}

\author{R. Komendarczyk} 
\thanks{The second author acknowledges support of DARPA YFA N66001-11-1-4132 and NSF DMS 1043009.}
\address{
Tulane University,
New Orleans, Louisiana 70118 } 
\email{rako@tulane.edu}

\author{C. Shonkwiler} 
\address{
 University of Georgia,
 Athens, Georgia 30602} 
 \email{clayton@math.uga.edu}

\subjclass[2010]{Primary: 57M25, 55Q25 Secondary: 57M27}

\date{\today}

\dedicatory{}

\commby{}

\begin{abstract}
 We provide an alternative proof that Koschorke's $\kappa$-invariant is injective on the set of link homotopy classes of $n$-component homotopy Brunnian links $BLM(n)$. The existing proof (by Koschorke \cite{Koschorke97}) is based on the Pontryagin--Thom theory of framed cobordisms, whereas ours is closer in spirit to techniques based on Habegger and Lin's string links. 
We frame the result in the language of Fox's torus homotopy groups and the rational homotopy Lie algebra $\pi_\ast(\Omega\Cnf(n))\otimes \Q$ of the configuration space. It allows us to express the relevant Milnor's $\mu$--invariants as homotopy periods of $\Cnf(n)$.
\end{abstract}

\maketitle

\addtocounter{footnote}{1}

\vspace{-.5cm}
\section{Introduction}\label{sec:intro}
\no The purpose of this paper is to develop methods to address a question of Koschorke \cite[p. 315]{Koschorke97} concerning classical smooth links in $\R^3$. Namely, Koschorke defines a link invariant, called the $\kappa$--invariant, by:
(i) defining a map from the space of {\em link maps} to the space of continuous maps of a torus to a configuration space, and
(ii) passing to the path-components.  He conjectures that this invariant distinguishes links up to link homotopy.
The tools developed in this paper are a computation of a so-called torus homotopy group which provides a framework for studying link homotopy classes and Koschorke's invariant.
The group in question (denoted later by $TF(n)$) is a natural subset of the pointed homotopy classes of maps of a torus to the configuration space. The resulting group is described by generators and relations
as well as associated information. The relations are built up out of relations which appear in a seemingly different context which first arose in work of Kohno and Drinfel'd on a
different subject arising from monodromy representations from the {\em KZ--equation}, (see \cite{Chari-Pressley-book94} for a comprehensive reference). 
In particular, it is shown that  the group obtained here for addressing Koschorke's conjecture is generated by certain elements in a function space. Furthermore, the group generated by
these natural choices satisfies a group theoretic analogue of relations known as the {\em Yang--Baxter relations} or {\em horizontal 4T relations} within knot theory.

\subsection{Preliminaries}\label{sec:prelim}
Let us fix $n$ distinct points $\{x_1,\ldots,x_n\}$ in $\R^3$ and  consider the space of (based) $n$--component link maps, i.e. smooth maps, 
\begin{equation}\label{eq:link-map}
\begin{split}
 L:& (S^1,s_1)\sqcup\ldots \sqcup (S^1,s_n)  \xrightarrow{\qquad} \R^3,\\
 & L(s_i)=x_i,\qquad L_i(S^1)\cap L_j(S^1)= \emptyset,\ i\neq j,
\end{split}
\end{equation}
where each component has a basepoint $s_i$ mapped to a corresponding fixed point $x_i$ in $\R^3$.  
Equivalently, we may consider the space of free link maps where the assumption on the basepoints is dropped, however as there is a bijective correspondence between the pointed and basepoint free theory by a standard argument. Two link maps $L$ and $L'$ are {\em link homotopic}  if and only if there exists a smooth homotopy 
$H:\left(\bigsqcup^n_{i=1} S^1\right)\times I\mapsto \R^3$ connecting $L$ and $L'$ through link maps. Following \cite{Koschorke97}, we denote the set of equivalence classes of $n$--component link maps by $LM(n)$. 
 By transversality, a link homotopy (as originally defined by Milnor \cite{Milnor54}) of a link map  can be realized at the level of diagrams by a finite sequence of Reidemeister moves and crossing changes \cite{Habegger-Lin90}; in particular the link homotopy classification of links is equivalent to the link homotopy classification of link maps.

In \cite{Milnor54}, Milnor classified $3$-component links up to link homotopy by the set of invariants built from the pairwise linking numbers $\bar{\mu}(1;2), \bar{\mu}(1;3), \bar{\mu}(2;3)\in \Z$ and a triple linking number $\bar{\mu}(1,2;3)\in \Z_{\gcd\{\bar{\mu}(1;2), \bar{\mu}(1;3), \bar{\mu}(2;3)\}}$. In particular he showed that the collection of integer valued $\mu$--invariants $\{\mu(1, \sigma(2)$, $\ldots$, $\sigma(n-1); n)(L)\}_\sigma$ indexed by permutations $\sigma\in \Sigma(2,\ldots,n-1)$ separates {\em homotopy Brunnian links} (following Milnor's lead, we use $\mu$ rather than $\bar\mu$ to indicate the lack of indeterminacy in the Brunnian case, see the end of Section \ref{S:proof} for detailed definitions). Recall that  an $n$-component link $L$ is \emph{homotopy Brunnian}~\cite{Koschorke97} whenever all of its $(n-1)$-component sublinks are link-homotopically trivial; we denote the subset of these links by $BLM(n)\subset LM(n)$.  It is well known \cite{Levine88, Hughes93} that the $\bar{\mu}$-invariants are insufficient to separate $LM(n)$ for $n \geq 4$, though a refinement of these invariants due to Levine \cite{Levine88} classifies links of four or fewer components. The question of classification has been effectively addressed in 1990 by
Habegger and Lin \cite{Habegger-Lin90} who gave an effective procedure for distinguishing elements in $LM(n)$ for all $n$. Nevertheless it still remains an open problem whether a complete set of ``numerical'' link homotopy invariants for $LM(n)$ can be defined. For instance, the authors of \cite{Mellor-Thurston01, Hughes03, Lin01} address this question using, among other things, the perspective of Vassiliev finite-type invariants. An alternative view, which is rarely cited in this context but which we intend to advocate here, appears in the work on higher dimensional link maps by Koschorke \cite{Koschorke90,Koschorke91b,Koschorke91,Koschorke97,Koschorke03}, Haefliger \cite{Haefliger62}, Massey and Rolfsen \cite{Massey-Rolfsen85}, and more recently by Munson \cite{Munson11} and Munson--Goodwillie in \cite{Goodwillie-Munson10}.

Following \cite{Koschorke91, Koschorke97}, the $\kappa$-invariant is defined for classical links (i.e., 1-dimensional links in $\R^3$) as
\begin{equation}\label{eq:kappa}
	\begin{split}
	\kappa: LM(n) & \to [\T^n, \Cnf(n)],\\
	\kappa([L]) & = [F_L], \qquad F_L = L_1 \times \cdots \times L_n,
	\end{split}
\end{equation}
where $[\T^n, \Cnf(n)]$ is the set of pointed homotopy classes of maps from the $n$-torus $\T^n$ to the configuration space $\Cnf(n)$ of $n$ distinct points in $\R^3$:
\[
	\Cnf(n) = \{(x_1, \ldots , x_n) \in (\R^3)^n\ |\ x_i \neq x_j \text{ for } i \neq j\}.
\]
Again, we consider pointed maps in \eqref{eq:kappa} purely for convenience.  Koschorke \cite{Koschorke91,Koschorke97,Koschorke91b,Koschorke90,Koschorke03} introduced the following central question which is directly related to the problems mentioned above and the $\kappa$--invariant.
\begin{prob}[Koschorke~\cite{Koschorke97}]\label{q:koschorke}
	Is the $\kappa$-invariant injective and therefore a complete invariant of $n$-component classical links up to link homotopy?
\end{prob}
\no In 1997 Koschorke showed that $\kappa$ is injective on $BLM(n)$ \cite[Theorem 6.1 and Corollary 6.2]{Koschorke97}; i.e., it separates homotopy Brunnian links. More recently the above question was answered in the affirmative when $n=3$, \cite{Kom-Milnor08, Kom-Milnor-degree09}. The work of Munson and Voli\'c (e.g. Proposition 3.10 \cite{Munson-Volic09}) gives a possibly stronger invariant of homotopy string links than the Koschorke invariant, an as yet unsettled question. Although their methods apply primarily to links in dimensions greater than $3$, these methods also provide a potentially interesting setting for links in $\R^3$.

\subsection{Statement of the main result}\label{sec:main-results}
In order to frame the main theorem in the language of the rational Lie algebra of $\Omega \Cnf(n)$, the based loops on the configuration space $\Cnf(n)$, we first introduce the necessary background. 
It is well known \cite{Cohen-Lada-May76, Fadell-Husseini01} that  $\pi_\ast(\Omega \Cnf(n)) \otimes \Q$, which  we denote further by $\mathcal{L}(\Cnf(n))$, is a graded Lie algebra with the bracket given by the Samelson product which is the adjoint of the usual Whitehead product \cite{Whitehead78}. More precisely, for any space $X$  the Samelson product of $\alpha \in \pi_k(\Omega X)$ and $\beta \in \pi_j(\Omega X)$ is given by 
\[
 [\alpha, \beta ] = (-1)^{k-1} \partial_\ast [\partial_\ast^{-1} \alpha, \partial_\ast^{-1} \beta]_W,
\]
 where $\partial_\ast: \pi_{p+1}(X) \to \pi_p(\Omega X)$ is the adjoint homomorphism and $[\cdot , \cdot]_W$ is the Whitehead product.
The generators of $\mathcal{L}(\Cnf(n))$ are all in degree $1$ and are represented by maps $B_{j,i}: S^1 \to \Omega\Cnf(n)$, $1 \leq i < j \leq n$, defined as adjoints of pointed versions of the spherical cycles
\begin{equation}\label{eq:A-ij}
		A_{j,i} : S^2 = \Sigma S^1  \longrightarrow \Cnf(n),\qquad
		A_{j,i}(\xi)  = (\ \ldots ,\underbrace{q_j}_{i\text{'th}},  \ldots , \underbrace{q_j+\xi}_{j\text{'th}},\ldots\ ), 
\end{equation}
where $\xi\in S^2$ and $q = (q_1, \ldots , q_n)$ is fixed in $\Cnf(n)$. We also have the following vector space isomorphism \cite[p.~22]{Fadell-Husseini01}:
\begin{equation}\label{eq:L(Conf(n))}
	\mathcal{L}(\Cnf(n)) \stackrel{\text{vect.}}{\cong} \bigoplus_{j=1}^{n-1} \pi_\ast (\Omega(\underbrace{S^2 \vee \ldots \vee S^2}_{j \text{ times}})) \otimes \Q = \bigoplus_{j=2}^{n} \mathcal{L}(B_{j,1}, \ldots, B_{j,j-1}),
\end{equation}
where the $j$th factor $\mathcal{L}(B_{j,1}, \ldots, B_{j,j-1})$ equals $\pi_\ast(\Omega(S^2 \vee \ldots \vee S^2))\, \otimes\, \Q$ and is the free Lie algebra generated by $\{B_{j,k}\}$, $k=1\ldots j-1$.  As a Lie algebra, $\mathcal{L}(\Cnf(n))$ is the quotient of the  direct sum of the free Lie algebras $\mathcal{L}_j$ by the {\em $4T$-relations} \cite[p. 20]{Fadell-Husseini01}, \cite{Kohno02, Cohen-Lada-May76}:
\begin{equation} \label{eq:4T-rel}
\begin{split}
  B_{i,j}  =-B_{j,i}\quad [B_{\sigma(2),\sigma(1)},B_{\sigma(4),\sigma(3)}] & =0,  \qquad (\text{for}\ n\geq 4),\\
 [B_{\sigma(2),\sigma(1)},B_{\sigma(3),\sigma(1)}+B_{\sigma(3),\sigma(2)}] & =  0,
\end{split}
\end{equation} 
where $\sigma$ is any permutation on $\{1,2,\ldots,n\}$. We can now state the main result.
\begin{mthm}%
	The restriction of $\kappa$ to $BLM(n)$ is injective.  Moreover,
	\begin{enumerate}
		\renewcommand{\theenumi}{(\roman{enumi})}
		\renewcommand{\labelenumi}{\theenumi}
		\item \label{item:kappa-injective} the image $\kappa(BLM(n))$ of $BLM(n)$ is contained in a copy of $\pi_n(\Cnf(n))\cong \pi_{n-1}(\Omega\Cnf(n))$ inside $[\T^n, \Cnf(n)]$ and it is a free, rank $(n-2)!$, $\Z$--module generated by 
		\begin{equation}\label{eq:B-sig}
			B(n,\sigma)=[B_{n,1}, B_{n, \sigma(2)}, \ldots , B_{n, \sigma(n-1)}], \quad \sigma \in \Sigma(2, \ldots , n-1);
		\end{equation}
		where $[B_{n,1}, B_{n, \sigma(2)}, \ldots , B_{n, \sigma(n-1)}]$ is a shorthand for the iterated Samelson products
		\[
		 [\ldots[[B_{n,1}, B_{n, \sigma(2)}],B_{n, \sigma(3)}], \ldots , B_{n, \sigma(n-1)}].
		\]
	    \item \label{item:kappa-mu} for any representative link $L \in BLM(n)$ we have the following expansion in the above basis:
		\[
			\kappa(L) = \sum_{\sigma \in \Sigma(2,\ldots , n-1)} \mu(1, \sigma(2), \ldots , \sigma(n-1); n)B(n,\sigma),
		\]
		where $\mu(1, \sigma(2), \ldots , \sigma(n-1); n)$ are Milnor's $\mu$--invariants of Brunnian links.
	\end{enumerate}
\end{mthm}
Recall \cite{Sullivan77, Hain84} that the {\em homotopy periods} of a simply connected manifold $M$ (with finite Betti numbers) are integrals in the differential forms on $M$ which  detect all nontrivial elements of $\pi_\ast(M)\otimes\Q$. It is well known that the homotopy periods problem is completely solvable; see \cite{Sullivan77, Hain84} and the recent work \cite{Sinha-Walter08}. As a direct consequence of \ref{item:kappa-mu} we obtain
\begin{cor}\label{cor:homotopy-periods}
		Given any $L\in BLM(n)$, the associated $\mu$-invariants $\{\mu(1, \sigma(2)$, $\ldots$, $\sigma(n-1); n)(L)\}_\sigma$, $\sigma\in \Sigma(2,\ldots,n-1)$ are fully determined by the homotopy periods of the basis elements $B(n,\sigma)$.
\end{cor}

 The inspiration for the proof of the Main Theorem comes from the algebraic techniques introduced by the first author in \cite{Cohen05} and from the second  proof in \cite[Section~5]{Kom-Milnor08}. Moreover, Corollary~\ref{cor:homotopy-periods} implies that the $\mu$-invariants of homotopy Brunnian links can be computed by Chen's iterated integrals \cite{Hain84}, which in this light appear as generalized Gauss integrals and hence as a possible source of invariants for fluid flows \cite[p.~176]{ArnoldKhesin} (see also \cite{Kom-higher-helicity2}).

{\subsection*{Acknowledgments}
{\small We would like to thank the University of Rochester for hosting our visit in July, 2010. The second and third authors are grateful to Dennis DeTurck, Herman Gluck, Paul Melvin, Jim Stasheff, and David Shea Vela-Vick for inspirational weekly meetings during 2006--2009. 
The first and second author are grateful for the hospitality of {\em The Kavli Institute for Theoretical Physics}, supported in part by the NSF PHY11-25915, and the organizers of the program {\em Knotted Fields, 2012}. }

\section{String links}\label{S:string-links}

Following Habegger and Lin \cite{Habegger-Lin90}, let $\mathcal{H}(n)$ be the group of link homotopy classes of ordered, oriented string links with $n$ components. There is a split short exact sequence of groups
\begin{equation}\label{eq:H(n)-short-exact}
	1 \longrightarrow \mathcal{C}(n;i) \xrightarrow{\quad\quad} \mathcal{H}(n) \xrightarrow{\quad\delta_i\quad}\mathcal{H}(n;i)\longrightarrow 1
\end{equation}
where $\mathcal{H}(n;i)$ is the copy of $\mathcal{H}(n-1)$ given by the map $\delta_i$ which deletes the $i$th strand. The normal subgroup $\mathcal{C}(n;i)$ is isomorphic to $RF(n-1)$, the reduced free group on the $n-1$ generators shown in Figure~\ref{fig:tau-nj}. Recall that the reduced free group $RF(n-1)$ is the quotient of the free group $F(n-1)$ on the $(n-1)$ generators $\tau_1, \ldots , \tau_{n-1}$ obtained by  adding the relations $[\tau_j, g \tau_j g^{-1}] = 1$ for all  $j \in \{1, \ldots , n-1\}$ and all $g \in F(n-1)$. 
\begin{figure}[htbp]
    \begin{center}
     \includegraphics[width=0.37\textwidth]{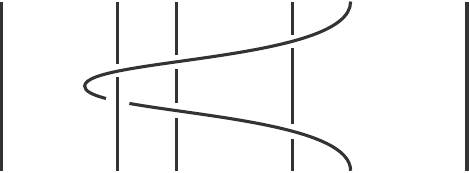}   
		\put(-135.5,-8){\small $1$}
		\put(-103.5,-8){\small $j$}
		\put(-36,-8){\small $i$}
		\put(-4,-8){\small $n$}
		\put(-128,20){\LARGE \ldots}
		\put(-76,20){\LARGE \ldots}
		\put(-31.5,20){\LARGE \ldots}
\end{center}
     \caption{{\small A generator $\tau_{i,j}$ of $\mathcal{C}(n;i)\cong RF(n-1)$.}}\label{fig:tau-nj}
\end{figure}
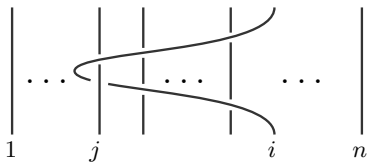

The exact sequence \eqref{eq:H(n)-short-exact} is split by the map $s_i:\mathcal{H}(n;i)\longrightarrow \mathcal{H}(n)$ which just adds one trivial strand to any element of $\mathcal{H}(n;i)$ as the $i$th strand. Thus, for each $i$, we have the semidirect product decomposition
\begin{equation}\label{eq:reduced-decomp}
	\mathcal{H}(n) = \mathcal{H}(n;i) \ltimes \mathcal{C}(n;i).
\end{equation}
Given $\rho \in \mathcal{H}(n)$, let
\begin{equation}\label{eq:factorization}
\rho=(\theta, h)_i=\theta_i h_i,\qquad  \theta_i\in\mathcal{H}(n;i) \quad \text{and}\quad h_i\in \mathcal{C}(n;i),
\end{equation}
be the factorization of $\rho$ with respect to the above decomposition. A \emph{partial conjugation} of $\rho$ by $\lambda \in \mathcal{H}(n)$ is obtained by $\rho = (\theta,h)_i \mapsto (\theta, \lambda h \lambda^{-1})_i$. 
\begin{thm}[Habegger--Lin Markov-type theorem \cite{Habegger-Lin90}]\label{thm:habegger-lin}
	Denote the Markov closure operation by
	\begin{equation}\label{eq:markov-closure}
		\widehat{\cdot}: \mathcal{H}(n) \xrightarrow{\qquad\quad} LM(n).
	\end{equation}
	Then
	\begin{enumerate}
		\renewcommand{\theenumi}{(\alph{enumi})}
		\renewcommand{\labelenumi}{\theenumi}
		\item \label{item:HL-surjective} $\widehat{\cdot}$ is surjective.
		\item \label{item:HL-preimage} $\widehat{\rho}_1 = \widehat{\rho}_2$ if and only if $\rho_1$ and $\rho_2$ in $\mathcal{H}(n)$ are related by a sequence of conjugations and partial conjugations (in fact, partial conjugations are sufficient \cite{Hughes05}).
	\end{enumerate}
\end{thm}
\no We collect some known facts about $\mathcal{H}(n)$ below. Here and throughout the paper, if $G$ is a group then its $k$th lower central subgroup is defined inductively by $G_1 = G, G_2=[G,G_1],\ldots , G_k = [G,G_{k-1}]$.
\begin{enumerate}
	\renewcommand{\theenumi}{\bf(\arabic{enumi})}
	\renewcommand{\labelenumi}{\theenumi}
 	\item \label{item:torsion-free} $\mathcal{H}(n)$ is torsion-free and nilpotent of class $n-1$,
 	\item \label{item:semidirect-LCS} $\mathcal{H}(n)_k=\mathcal{H}(n;i)_k\ltimes \mathcal{C}(n;i)_k$,
 	\item \label{item:H(n)-kth-LCS} $\mathcal{H}(n)_{k-1}/\mathcal{H}(n)_k$ is a free abelian group of rank $(k-2)! {n\choose k}$.
\end{enumerate} 

In the following we will focus on the copy of $\mathcal{C}(n;n) \cong RF(n-1)$ in $\mathcal{H}(n)$ generated by $\langle \tau_1, \ldots , \tau_{n-1}\rangle$ where $\tau_k := \tau_{n,k}$, as given in Figure \ref{fig:tau-nj}. From \cite{Cohen05} we list known and useful facts about $RF(n-1)$ below.
\begin{enumerate}
	\renewcommand{\theenumi}{\bf(\arabic{enumi})}
	\renewcommand{\labelenumi}{\theenumi}
	\setcounter{enumi}{3}
	\item \label{item:RF-presentation} $RF(n-1)$ admits the presentation
	\begin{equation}\label{eq:RF-relations}
		\begin{split}
			& \Bigl\langle \tau_1, \ldots , \tau_{n-1} \ \bigl|\ \text{for all } 1 \leq i_1 < \ldots < i_k \leq n, 1<k\leq n:\\
			& \qquad \qquad \qquad [\tau_{i_1}, \ldots , \tau_{i_k}] = 1, \text{ whenever } \tau_{i_s} = \tau_{i_r} \text{ for some } s < r \Bigr\rangle,
		\end{split}
	\end{equation}
	where $[\tau_{i_1}, \ldots , \tau_{i_k}]$ denotes the simple $k$-fold commutator $[\cdots [[\tau_{i_1}$, $\tau_{i_2}]$, $\tau_{i_3}]$, $\ldots$, $\tau_{i_k}]$ (c.f. \cite[p. 295]{Magnus-Karrass-Solitar76}).
\end{enumerate}

\no Denote by $I = (i_1, \ldots , i_k)$ an \emph{ordered} multiindex, where $1 \leq i_1 < \ldots < i_k \leq n-1$, $1 \leq k \leq n-1$, and let
\begin{equation}\label{eq:tau-I-sigma}
	\tau(I,\sigma) := [\tau_{i_1},\tau_{i_{\sigma(2)}}, \ldots , \tau_{i_{\sigma(k)}}],
\end{equation}
where $\sigma$ is a permutation of $\{2, \ldots , k\}$. Then
\begin{enumerate}
	\renewcommand{\theenumi}{\bf(\arabic{enumi})}
	\renewcommand{\labelenumi}{\theenumi}
	\setcounter{enumi}{4}
	\item $RF(n-1)_n = \{1\}$ and
	\[
		RF(n-1)_k\Bigl/RF(n-1)_{k+1} \cong \bigoplus_{(k-1)! {n-1 \choose k}} \Z.
	\]
	\item \label{item:RF-normal} Each $RF(n-1)_k\Bigl/RF(n-1)_{k+1}$ is generated by $\{\tau(I,\sigma)\}$, with $|I| = k$, and elements $z \in RF(n-1)$ have the normal form
	\begin{equation}\label{eq:normal-form-RF}
		z = \lambda_1 \lambda_2 \cdots \lambda_{n-1}, \quad \text{where}\quad \lambda_k = \prod_{I, |I| = k} \prod_{\sigma \in \Sigma(2, \ldots , k)} \tau(I, \sigma)^{e(I, \sigma)},
	\end{equation}
for some $e(I,\sigma) \in \Z$.
\end{enumerate}

\no Consider the homomorphism
\begin{equation}\label{eq:delta}
	\delta: \mathcal{H}(n) \xrightarrow{\qquad \quad} \prod^n_{i=1} \mathcal{H}(n;i), \qquad \delta = \prod^n_{i=1} \delta_i,
\end{equation}
where $\delta_i$ is as defined in \eqref{eq:H(n)-short-exact}. Clearly,
\begin{equation}\label{eq:ker-delta}
	\ker \delta = \mathcal{C}(n;1) \cap \mathcal{C}(n;2) \cap \ldots \cap \mathcal{C}(n;n).
\end{equation}
Observe that elements of $\ker \delta$ have a natural geometric meaning:  they are precisely the string links which become trivial after removing \emph{any} of their components. Further we call them {\em Brunnian string links} and denote by $\mathcal{BH}(n)$.
In the ensuing lemma we choose to treat $\mathcal{BH}(n)$ as a subgroup of $\mathcal{C}(n;n) \cong RF(\tau_1, \ldots , \tau_{n-1})$.
\begin{lem}\label{lem:BH(n)}
	$\mathcal{BH}(n)$ is a free abelian group of rank $(n-2)!$ generated by
	\begin{equation}\label{eq:tau-sig}
		\tau(n,\sigma) := [\tau_1, \tau_{\sigma(2)}, \ldots , \tau_{\sigma(n-1)}] \quad \text{for} \quad \sigma \in \Sigma(2, \ldots , n-1).
	\end{equation}
	Moreover,
	\begin{enumerate}
		\renewcommand{\theenumi}{(\roman{enumi})}
		\renewcommand{\labelenumi}{\theenumi}
		\item \label{item:BH(n)-C(n;i)} for each $i$:
		\[
		 \mathcal{BH}(n) =\mathcal{C}(n;i)_{n-1}\cong \mathcal{H}(n)_{n-1} \cong RF(n-1)_{n-1};
		\]
		\item \label{item:BH(n)-Z(H(n))} $\mathcal{BH}(n)\subset Z(\mathcal{H}(n))$.
	\end{enumerate}
\end{lem}
\begin{proof}
	The fact that $\mathcal{H}(n)_{n-1} \subset Z(\mathcal{H}(n))$ follows immediately from nilpotency (i.e. length $n$ commutators are all trivial in $\mathcal{H}(n)$). 
Thus \ref{item:BH(n)-C(n;i)} implies \ref{item:BH(n)-Z(H(n))}. 

Clearly $\mathcal{BH}(n)\subset \ker \delta_i=\mathcal{C}(n;i)$, so any $z\in \mathcal{BH}(n)$ can be written in the normal form \eqref{eq:normal-form-RF}. We claim that $z=\lambda_{n-1}$. Indeed, for each $k<n-1$ and any $I = (i_1, \ldots , i_k)$, consider $\tau(I,\sigma)$ given in \eqref{eq:tau-I-sigma}. Pick $j \in \{1,\ldots ,  n-1\}$ such that $j\neq i_r$  for all $i_r\in I$, which is possible since $k<n-1$. The map
 $\delta_j: \mathcal{C}(n;i) \longrightarrow \mathcal{C}(n;i)\cap \mathcal{C}(n;j)$ which deletes the $j$th strand is given on generators as
	\begin{equation}\label{eq:generator-delta}
		\delta_j(\tau_{i}) = \begin{cases} 1 & \quad \text{for } j = i,\\
      \tau_{i} & \quad \text{otherwise}, \end{cases}
	\end{equation} 
so we have $\delta_j(\tau(I,\sigma))=\tau(I,\sigma)$. Therefore, $\delta_j(\lambda_k) \neq 1$ for some $j\in \{1,\ldots,n-1\}$, contradicting the fact that $z \in \mathcal{BH}(n) \subset \ker \delta_j$ for all $j$. Hence, the normal form \eqref{eq:normal-form-RF} implies  $z=\lambda_{n-1}$ and therefore \eqref{eq:tau-sig} follows as well as the first part of \ref{item:BH(n)-C(n;i)}. 

The second identity in \ref{item:BH(n)-C(n;i)} is immediate: \ref{item:semidirect-LCS} implies that $\mathcal{H}(n)_{n-1}=\mathcal{H}(n;i)_{n-1}\ltimes \mathcal{C}(n;i)_{n-1}$, but this is just $\mathcal{C}(n;i)_{n-1}$ since $\mathcal{H}(n;i)_{n-1}\cong \mathcal{H}(n-1)_{n-1}$ is trivial by \ref{item:torsion-free}.
\end{proof}
 The relation between Brunnian string links and Brunnian links is revealed in the following result, which is a consequence of Lemma~\ref{lem:BH(n)} and Theorem~\ref{thm:habegger-lin}.
\begin{prop}\label{prop:BLM(n)}
	The restriction of the Markov closure operation defined in \eqref{eq:markov-closure} to $\mathcal{BH}(n)$ is injective, and the image $\widehat{\mathcal{BH}(n)}$ equals $BLM(n)$.
\end{prop}
\begin{proof}
	In order to see that\, $\widehat{\cdot}$\, is injective on $\mathcal{BH}(n)$, let $\rho \in \mathcal{BH}(n)$. Since $\mathcal{BH}(n) = \mathcal{C}(n;i)_{n-1}$, the factorization 
	\[
		\rho = (\theta,h)_i = \theta_i h_i
	\] 
	is only valid if $\theta_i = 1$. Since this holds for all $i$, partial conjugations of $\rho$ are just ordinary conjugations which, since $\mathcal{BH}(n)\subset Z(\mathcal{H}(n))$, act trivially on $\rho$. Thus, Markov closure is injective on $\mathcal{BH}(n)$.
	
In order to see $\widehat{\mathcal{BH}(n)}=BLM(n)$, observe that for any natural number $m$ the inverse image of the unlink in $LM(m)$ under\, $\widehat{\cdot}:\mathcal{H}(m)\longrightarrow LM(m)$ contains only $1\in \mathcal{H}(m)$.
Indeed, Theorem \ref{thm:habegger-lin}\ref{item:HL-preimage} tells us that any element of the inverse image of the unlink under\, $\widehat{\cdot}$\, has to be related to $1$ by conjugations or partial conjugations. Obviously, both of these operations act trivially on $1$, which proves the claim.

Now fix a representative link $L \in BLM(n)$ and, by Theorem~\ref{thm:habegger-lin}\ref{item:HL-surjective}, let $\rho \in \mathcal{H}(n)$ be such that $\widehat{\rho} = L$. Since any $(n-1)$-component sublink $\widehat{\delta_i(\rho)}$ of $L$ is trivial, the fact proven above implies that $\delta_i(\rho) = 1$ for any $1 \leq i \leq n$, and therefore $\rho \in \ker \delta = \mathcal{BH}(n)$.
\end{proof}

\section{Torus homotopy groups}\label{S:torus} 

In this and the following sections all sets of homotopy classes $[X,Y]$ are pointed and $\ast$ denotes a basepoint. Consider the group
\begin{equation}\label{eq:T(n)}
	T(n) := [\Sigma\T^{n-1}, \Cnf(n)] = [\T^{n-1}, \Omega \Cnf(n)]
\end{equation}
of homotopy classes of pointed maps from the $(n-1)$-torus $\T^{n-1} = (S^1)^{n-1}$ to the based loop space $\Omega \Cnf(n)$ of the configuration space $\Cnf(n)$. The product in $T(n)$ comes from the loop multiplication or equivalently the coproduct of suspensions. In the notation of Fox \cite{Fox45, Fox48}, who introduced torus homotopy groups, $T(n)$ is denoted by 
$\tau_n(\Cnf(n))$ and equivalently defined as $\pi_1(\text{Maps}(\T^{n-1},\Cnf(n)),\ast)$, where the basepoint $\ast$ in the case of $T(n)$ is defined to be the constant map. In the following, we will freely alternate between both ways of representing $T(n)$ given in \eqref{eq:T(n)}.

Letting $P_k$ be the $k$-skeleton of $\T^{n-1}$, we have the filtration
\[
	\{s\} = P_0 \subset P_1 \subset \ldots \subset P_{n-1} = \T^{n-1}, \qquad \ast = s = (s_1, \ldots , s_{n-1}).
\]
For each ordered multiindex $I$, observe that  $P_k = \bigcup_{I, |I| = k} S_I$, with notation 
\begin{equation}\label{eq:S_I-S^I}
\begin{split}
	S_I & = \{t = (t_1, \ldots , t_{n-1}) \in \T^{n-1}\, | \, t_i = s_i \text{ for } i \notin I\}.\\
	S^I & = S_I/(S_I \cap P_{|I|-1}).
\end{split}	
\end{equation}
It implies a decreasing  filtration of groups
\[
	T(n) \supset T(n;1) \supset \cdots \supset T(n;n-1) = \{1\},
\]
where $T(n;k) := [(\T^{n-1}, P_k) ; (\Omega \Cnf(n), \ast)]$.
\no Some known facts about $\{T(n;k)\}$ are summarized below (c.f. \cite[p. 462]{Whitehead78}):
\begin{enumerate}
	\renewcommand{\theenumi}{\bf(\arabic{enumi})}
	\renewcommand{\labelenumi}{\theenumi}
	\setcounter{enumi}{6}
	\item $\{T(n;k)\}$ is a central chain of $T(n)$.
	\item $T(n)$ is nilpotent of class $n-1$.
	\item \label{item:T-LCS-quotient}
	\begin{equation}\label{eq:gamma-quotient}
		T(n;k-1) \, \Bigl/\, T(n;k) \cong \bigoplus_{I,\ |I| = k-1} \pi_I(\Omega\Cnf(n)),
	\end{equation}
	where $\pi_I(\Omega\Cnf(n)) := [S^I, \Omega\Cnf(n)]$.
\end{enumerate}
\no Each factor $\pi_I(\Omega\Cnf(n))$ in the direct sum of \eqref{eq:gamma-quotient} turns out to be a subgroup of $T(n)$ via a monomorphism indexed by $I = (i_1, \ldots , i_k)$, $1\leq i_1 < \ldots<i_k\leq n-1$ defined as follows (c.f. \cite{Fox45, Whitehead78}):
\begin{equation}\label{eq:j_I}
	\begin{split}
		j^\#_I & : [S^I, \Omega\Cnf(n)] \xrightarrow{\qquad\quad} T(n),\\
		& j^\#_I([f]) = [f \circ j_I], \quad j_I: \T^{n-1}\! \longrightarrow S^I,
	\end{split}
\end{equation}
where $j_I$ is the quotient projection.
\begin{rem}
{\em
\no Another way to see \eqref{eq:gamma-quotient} is via the homotopy equivalence 
$\Sigma \T^{n-1}\cong\bigvee^{n-1}_{k=1}\ \bigvee_{I,\,|I|=k} \Sigma S^k$.
Then each monomorphism $j^\#_I$ of \eqref{eq:j_I} can be regarded as induced from the restriction in the above
bouquet to the $I$th factor $\Sigma S^{|I|}$ of $\Sigma \T^{n-1}$. 
}
\end{rem}
We have the  following lemma due to Fox \cite[p. 498]{Fox48}:

\begin{lem}\label{lem:fox} For every ordered multiindex $I$, $j^\#_I$ is a monomorphism. For $I$, $J$:
	\begin{enumerate}
		\renewcommand{\theenumi}{(\roman{enumi})}
		\renewcommand{\labelenumi}{\theenumi}
		\item \label{item:fox-IJ-empty} Suppose $I \cap J = \text{\rm \O}$ and $\epsilon = (-1)^w$, where $w$ is the number of instances of $i > j$ with $i \in I$ and $j \in J$. Then for any $\alpha \in \pi_I(\Omega\Cnf(n)), \beta \in \pi_J(\Omega\Cnf(n))$, we have
		\[
			j^\#_{I \cup J} ([\alpha, \beta]) = [j^\#_I(\alpha), j^\#_J(\beta)]^\epsilon,
		\]
		where $[\alpha, \beta]$ is the Samelson product of $\alpha$ and $\beta$, and $[\,,\,]$ on the right hand side denotes the commutator in $T(n)$.
		\item \label{item:fox-IJ-nonempty} If $I \cap J \neq \text{\rm \O}$, then 
			$[j^\#_I(\alpha), j^\#_J(\beta)] = 1$.
	\end{enumerate}
\end{lem}
\no  Consider the following elements of $T(n)$ obtained from the generators $\{B_{k,i}\}$ of $\mathcal{L}(\Cnf(n))$ which are adjoints of the $\{A_{k,i}\}$ defined in \eqref{eq:A-ij}:
\begin{equation}\label{eq:torus-free-gen}
	t_{k,i}(\ell) := j^\#_{(\ell)}(B_{k,i}) \in \pi_{(\ell)}(\Omega\Cnf(n)), \quad 1 \leq i < k \leq n, \quad 1 \leq \ell \leq n-1,
\end{equation}
where $(\ell)$ is a length $1$ multiindex. In view of the fact that $\{B_{k,i}\}$ are free generators of $\pi_1(\Omega\Cnf(n))$, the $t_{k,i}(\ell)$ each have infinite order and $\{t_{k,i}(\ell)\}$ generates a subgroup of $T(n)$ which we will denote by $TF(n)$. The following lemma is a consequence of Theorem~5.4 in \cite{Magnus-Karrass-Solitar76} and we omit its technical proof.
\begin{lem}\label{lem:TF(n)}
	The lower central series $\{TF(n)_k\}$ of $TF(n)$ has the following properties:
	\begin{enumerate}
		\renewcommand{\theenumi}{\bf(\arabic{enumi})}
		\renewcommand{\labelenumi}{\theenumi}
		\setcounter{enumi}{9}
		\item \label{item:TF-nilpotent} $TF(n)$ is nilpotent of class $n-1$, and 
		\item \label{item:TF(n)-k-subset-T(n)-k} $TF(n)_k \subset T(n;k)$ for every $k$, and the last stage $TF(n)_{n-1}$ is generated by length $n-1$ commutators in $\{t_{k,i}(\ell)\}$.
	\end{enumerate}
\end{lem}
\no Next, let us define and characterize the \emph{Brunnian part} of $TF(n)$, which we denote by $BTF(n)$. Consider the following projection map, which is defined by analogy to the string link story as
\[
	\Omega\psi: \Omega\Cnf(n) \xrightarrow{\quad \prod^n_{i=1} \Omega\psi_i\quad } \prod_n \Omega\Cnf(n-1),
\]
obtained by looping $\psi = \psi_1 \times \ldots \times \psi_n$, where $\psi_i: \Cnf(n) \to \Cnf(n-1)$ is the projection $\psi_i: (x_1, \ldots , x_i , \ldots , x_n) \mapsto (x_1, \ldots \widehat{x}_i, \ldots , x_n)$
``deleting'' the $i$th coordinate. As in the case of string links, we say that $t \in TF(n)$ is \emph{Brunnian} whenever $t$ belongs to the kernel of the homomorphism
\[
	\Psi: TF(n) \xrightarrow{\quad} \prod_n [\T^{n-1}, \Omega\Cnf(n-1)],
\]
induced from the product $\prod^{n-1}_{i=1}  \Omega\psi_i$.
The following result should seem like \emph{d\`ej\'a vu} of Lemma \ref{lem:BH(n)} from the previous section. 
\begin{lem}\label{lem:Lie(n-1)}
	The group $BTF(n) = \ker \Psi$ is a free abelian group of rank $(n-2)!$ generated by iterated commutators in $TF(n)$ of the form
	\begin{equation}\label{eq:t-sig}
		t(n,\sigma) = [t_{n,1}(1), t_{n,\sigma(2)}(\sigma(2)), \ldots , t_{n, \sigma(n-1)}(\sigma(n-1))],
	\end{equation}
	where $\sigma \in \Sigma(2, \ldots , n-1)$ is any permutation of $\{2, \ldots , n-1\}$. In particular, it follows that  
\begin{equation}\label{eq:BTF-pi}
	BTF(n) \subset TF(n)_{n-1} \subset \pi_N(\Omega\Cnf(n)),
\end{equation}
where $N = (1,2,\ldots , n-1)$ is the top ordered index in the notation of \eqref{eq:j_I}.
\end{lem}

\begin{proof}
 Observe that $\Omega\psi_i^{\#}$ are defined on the generators $\{t_{k,j}(\ell)\}$ of $TF(n)$ as
	\begin{equation}\label{eq:generator-psi}
		\Omega\psi_i^{\#}(t_{k,j}(\ell)) = \begin{cases} 1 & \quad \text{for } k = i \text{ or } j = i \\
		t_{k,j}(\ell) & \quad \text{otherwise}. \end{cases}
	\end{equation}
	Each $\psi_i$ is a fibration, with the fiber having the homotopy type of $\bigvee_{i=1}^{n-1} S^2$, \cite[p. 14]{Fadell-Husseini01, Cohen-Lada-May76} where each $S^2$ factor corresponds to a generator in $\{B_{i,k},B_{r,i}\}_{k,r}$, it follows that 
$\{B_{i,k},B_{r,i}\}$ generate $\ker \pi_\ast(\Omega\psi_i)$, thus \eqref{eq:torus-free-gen} implies  \eqref{eq:generator-psi}.

Next, we show that every $t \in BTF(n)$ is a product of length $n-1$ commutators in the $t_{k,i}(\ell)$ which do not have repeated lower indices $\{k,i\}$. 
Similar to the situation in the proof of Lemma~\ref{lem:BH(n)}, no nontrivial commutator $[t_{k_1,i_1}(\ell_1)$, $\ldots$, $t_{k_j, i_j}(\ell_j)]$ of length $< n-1$ can be in $BTF(n)$. Indeed, any such commutator has to be in the image under $j^\#_I$, with $I = (\ell_1, \ldots , \ell_j)$, $1\leq \ell_1, \ldots , \ell_j\leq n-1$, of the iterated Samelson product $[B_{k_1,i_1}, \ldots, B_{k_j,i_j}]$. Thanks to the relations~\eqref{eq:4T-rel} and Lemma~\ref{lem:fox}, all the $B_{k_s,i_s}$ have to have a common first lower index $k$; thus, up to sign,  $[B_{k_1,i_1}, \ldots, B_{k_j,i_j}]$ equals $[B_{k,i_1}, \ldots, B_{k,i_j}]$. 
Furthermore, if $j < n-1$, then we can find $r$ with $1 \leq r \leq n-1$ such that $r \notin \{i_1, \ldots , i_j\}$. Then it follows from \eqref{eq:generator-psi} that 
\[
	\Omega\psi_r^{\#} \left([t_{k,i_1}(\ell_1), \ldots , t_{k,i_j}(\ell_j)]\right) \neq 1,
\]
so this commutator cannot be in $BTF(n)$.
Likewise, no nontrivial commutator $[t_{k,i_1}(\ell_1), \ldots, t_{k,i_{n-1}}(\ell_{n-1})]$, (where $(\ell_1$, $\ldots$, $\ell_{n-1})$ is a permutation of $(1,\ldots,n-1)$)  with repeated second lower indices (i.e., $i_p = i_q$ for some $p \neq q$) can be in $BTF(n)$. 
Thus, we are left only with the possibility that $t \in BTF(n)$ is the product of length $n-1$ commutators without repeated lower indices. Hence, \eqref{eq:BTF-pi} follows from \ref{item:TF(n)-k-subset-T(n)-k} of Lemma~\ref{lem:TF(n)}. 
It remains to observe that the $t(n, \sigma)$ defined in \eqref{eq:t-sig} generate $BTF(n)$ or, equivalently, that 
$\{B(n, \sigma) = [B_{n,1},B_{n,\sigma(2)}, \ldots , B_{n, \sigma(n-1)}]\}$
generate $\ker \pi_\ast(\prod_i\Omega \psi_i)$.  By abuse of notation we use $BTF(n)$ to denote $\ker \pi_\ast(\prod_i\Omega \psi_i)$ as well. Since we are only concerned with the free part of $\ker \Psi$ it suffices to work rationally and show the following about the $B(n, \sigma)$s:
\begin{enumerate}
	\renewcommand{\theenumi}{(\alph{enumi})}
	\renewcommand{\labelenumi}{\theenumi}
	\item \label{item:Bnsigma-span-BTF} they span $BTF(n) \otimes \Q$,
	\item \label{item:Bnsigma-lin-indep} they are linearly independent in $\mathcal{L}(\Cnf(n)) = \pi_\ast(\Omega\Cnf(n)) \otimes \Q$.
\end{enumerate}
For  \ref{item:Bnsigma-span-BTF}, note that (by an argument analogous to \cite[p. 295]{Magnus-Karrass-Solitar76}) $BTF(n)\, \otimes\, \Q$ is spanned by the simple $(n-1)$-fold products $[B_{n,i_1}, B_{n,i_2},\ldots , B_{n, i_{n-1}}]$ with no repeated $B_{n,i_k}$. Then a simple induction involving the Jacobi identity proves the claim.  Part \ref{item:Bnsigma-lin-indep} can be argued by considering $\mathcal{L}(\Cnf(n))$ as a subalgebra of the universal enveloping algebra $U\mathcal{L}(\Cnf(n))$ and expanding $B(n, \sigma)$ in monomial basis of $U\mathcal{L}(\Cnf(n))$. Again an elementary induction argument leads to the claim.
\end{proof}
\section{Proof of Main Theorem}\label{S:proof}

Let $(0,\ldots , 0)$ be the basepoint of $\T^n$ and let each factor be parametrized by the unit interval. Distinguish the following subsets of $\T^n$:
\begin{equation}\label{eq:A-set}
	A_t := \T^{n-1} \times \{t\}, \qquad S_n := \{(0,\ldots , 0)\} \times S^1,
\end{equation}
and define the maps
\begin{equation}\label{eq:p}
	\begin{split}
		& p^{\#} : [\Sigma\T^{n-1}, \Cnf(n)] \longrightarrow [\T^n, \Cnf(n)],\qquad \text{where}\\
		& p: \T^n \longrightarrow \T^n / (A_0 \vee S_n) \cong \Sigma\T^{n-1}.
	\end{split}
\end{equation}
\begin{lem}\label{lem:inj-pi_n}
	Let $N = (1,2,\ldots , n-1)$ be the top ordered multiindex. Consider the composition
	\[
		\pi_N(\Omega\Cnf(n)) \xrightarrow{\ j^\#_N\ } T(n) \xrightarrow{\ p^{\#}\ } [\T^n, \Cnf(n)],
	\]
	where $j^\#_N$ is defined in \eqref{eq:j_I}. Then $p^{\#} \circ j^\#_N$ is injective.
\end{lem}
\no This result follows from Satz 12, Satz 20 in \cite{Puppe58} (see \cite[p. 305]{Koschorke97}).
Next, we turn to the proof of the Main Theorem.

	We will work with $\mathcal{C}(n;n)$, which is a copy of $RF(n-1)$ inside $\mathcal{H}(n)$, and begin by constructing a homomorphism 
	\begin{equation}\label{eq:phi}
		\phi: \mathcal{C}(n;n) \longrightarrow TF(n)
	\end{equation}
	via the canonical homomorphism $F(n-1) \longrightarrow TF(n)$ defined on the generators $\tau_1, \ldots , \tau_{n-1}$ of the free group by
	\[
		\phi: \tau_i  \mapsto t_{n,i}(i).
	\]
	Observe, by Lemma~\ref{lem:fox}\ref{item:fox-IJ-nonempty}, that any commutator in $\{t_{n,i}(i)\}$ with repeats is trivial. Therefore, as a direct consequence of the presentation in \ref{item:RF-presentation}, we can pass to the quotient and obtain a homomorphism ${\phi: \mathcal{C}(n;n) \cong RF(n-1) \longrightarrow TF(n)}$, as required. 
	
\no 	{\bf Fact:}  The following diagram commutes:
	\begin{equation}\label{diag:kappa-phi}
		\begin{tikzpicture}[baseline=(current bounding box.center),description/.style={fill=white,inner sep=2pt}] \matrix (m) [matrix of math nodes, row sep=3em, column sep=2.5em, text height=1.5ex, text depth=0.25ex] { \mathcal{C}(n;n) & & TF(n) \\
  			LM(n) &  & \text{$[\T^n,\Cnf(n)]$.} \\ };
			\path[->,font=\scriptsize]
			(m-1-1) edge node[auto] {$ \phi $} (m-1-3) 
			(m-1-3) edge node[auto] {$ p^{\#} $} (m-2-3)
			(m-1-1) edge node[auto] {$ \widehat{\cdot} $} (m-2-1) 
			(m-2-1) edge node[auto] {$ \kappa $} (m-2-3);
		\end{tikzpicture}
	\end{equation}
With this fact in hand consider the composition $\phi \circ \iota$ where  $\iota: \mathcal{BH}(n) \longrightarrow \mathcal{C}(n;n)$ is the inclusion monomorphism (see Lemma~\ref{lem:BH(n)}). The normal form of any $z \in \mathcal{BH}(n)$ is a product of terms $\tau(n, \sigma)$ as defined in \eqref{eq:tau-sig}; thus, by the definition of $\phi$ and \eqref{eq:t-sig}, we immediately obtain
\[
	\phi(\tau(n,\sigma)) = t(n, \sigma).
\]
Hence, Lemma~\ref{lem:Lie(n-1)} implies that $\phi \circ \iota$ is a monomorphism with image equal to $BTF(n)$. Further, the Fact stated above yields the commutative diagram
\begin{equation}\label{diag:big}
	\begin{tikzpicture}[baseline=(current bounding box.center),description/.style={fill=white,inner sep=2pt}]
		\matrix (m) [matrix of math nodes, row sep=2em, column sep=2.5em, text height=1.5ex, text depth=0.25ex]
		{  \mathcal{BH}(n) & & BTF(n)\subset\pi_N(\Omega\Cnf(n))\\
    	 & & TF(n)\\
  		BLM(n) &  & \text{$[\T^n,\Cnf(n)]$.} \\ };
		\path[->,font=\scriptsize]
		(m-1-1) edge node[auto] {$\phi\circ\iota$} (m-1-3)
		(m-1-3) edge node[auto] {$j^\#_N$} (m-2-3) 
		(m-2-3) edge node[auto] {$p^{\#}$} (m-3-3)
		(m-1-1) edge node[auto] {$\widehat{\cdot}$} (m-3-1) 
		(m-3-1) edge node[auto] {$\kappa$} (m-3-3);
	\end{tikzpicture}
\end{equation}

\no By Lemma~\ref{lem:inj-pi_n} the composition $p^{\#} \circ j^\#_N \circ \phi \circ \iota$ is injective. Further, by Proposition~\ref{prop:BLM(n)} the Markov closure map $\,\widehat{\cdot}\, $ is a bijection on $\mathcal{BH}(n)$. Therefore, the injectivity of $\kappa$ follows, proving \ref{item:kappa-injective} of the Main Theorem modulo the above Fact. 
We also have the set identity
\[
	\kappa(BLM(n)) = p^{\#} \circ j^\#_N(BTF(n)),
\]
which implies that $\kappa(BLM(n))$ has the structure of a Lie $\Z$-module with basis given by the iterated Samelson products $B(n,\sigma)$ from \eqref{eq:B-sig}.

\no It remains to prove the Fact stated above. Figure~\ref{fig:geometricCD} shows the intuitive idea -- which originated in Section~5 of \cite{Kom-Milnor08} -- behind the formal argument we present below.
\begin{figure}[htbp]
	\centering
		\includegraphics[scale=.7]{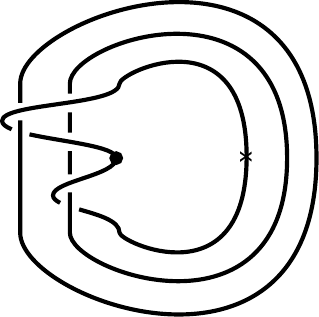} \hspace{.5in}
		\includegraphics[scale=.7]{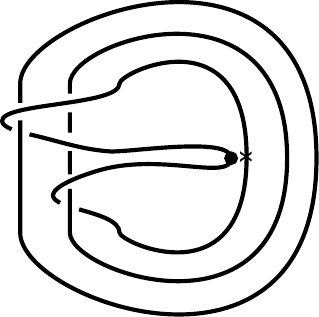}
	\caption{{\small The link $\widehat{\tau_{3,2}\tau_{3,1}}$ before and after an isotopy. The point indicated by $\bullet$ on the third component corresponds to the parameter $t_1$, while the basepoint of the third component is marked by $\times$. The isotopy moving $\bullet$ to the basepoint $\times$ makes it clear that the loop in\protect\footnotemark \  $\pi_1(\text{Maps}(\T^{n-1},\Cnf(n)),\ast)\cong T(n)$ corresponding to $\kappa(\widehat{\tau_{3,2}\tau_{3,1}})$ is homotopic to the product of the loops corresponding to $\kappa(\widehat{\tau_{3,2}})$ and $\kappa(\widehat{\tau_{3,1}})$; in other words, that $\kappa(\widehat{\tau_{3,2}\tau_{3,1}}) = p^\#(t_{3,2}(2)\cdot t_{3,1}(1)) = p^\#(\phi(\tau_{3,2}\tau_{3,1}))$. }
}
	\label{fig:geometricCD}
\end{figure}

\begin{proof}[Proof of the Fact]
		As a first step, we show commutativity of \eqref{diag:kappa-phi} on the generators $\{\tau_i\}$ of $\mathcal{C}(n;n)$. Given a multi-index $I$, we adapt the notation $S^I$ and $S_I$ from \eqref{eq:S_I-S^I}. 
Because all strands of $\widehat{\tau}_i$ except the $n$th and the $i$th can be collapsed to the basepoint (see Figure \ref{fig:tau-nj}), $\kappa(\widehat{\tau}_i)$ factors, up to homotopy, through 
\[
 \T^n\xrightarrow{\quad p_{(i,n)}\quad } S_{(i,n)}= S_i\times S_n\xrightarrow{\quad \kappa(\widehat{\tau}_i)|_{S_i\times S_n}\quad } \Cnf(n) .
\]
Since the $i$th and $n$th strands of $\tau_i$ link once, the restriction $\kappa(\widehat{\tau}_i)|_{S_i\times S_n}$ has degree $1$ after composing with the projection $\Pi_{i,n}:\Cnf(n)\longrightarrow \Cnf(2)\cong S^2$ onto the $i$th and $n$th coordinates. Further, $\kappa(\widehat{\tau}_i)|_{S_i\times S_n}$ factors 
through $S_i\times S_n \xrightarrow{\ p_i\ } \Sigma S_i \xrightarrow{\ A_{n,i}\ } \Cnf(n)$,
where the $A_{n,i}$ are defined in \eqref{eq:A-ij} (note that $p_i$ induces a bijection $p^\#_i:[\Sigma S_i,\Cnf(n)]\longrightarrow [S_i\times S_n,\Cnf(n)]$). Using the definitions in \eqref{eq:j_I}, \eqref{eq:torus-free-gen} and \eqref{eq:p} we have
	\begin{equation}\label{eq:CDstep1}
		\kappa(\widehat{\tau}_i) = p_{(i,n)}^{\#}(p_{i}^{\#}(A_{n,i}))  = p^{\#}(j^\#_{(i)}(B_{n,i})) = p^{\#}(t_{n,i}(i)),
	\end{equation}
\no 	where the second equality is obtained by passing to adjoints. This shows that Diagram~\ref{diag:kappa-phi} commutes on the generators.
\footnotetext{The basepoint $\ast$ of $\pi_1(\text{Maps}(\T^{n-1},\Cnf(n)),\ast)$ can be taken (w.l.o.g.) as the null map obtained from the restriction of the link to its trivial first $n-1$ strands.}
Now, let $\tau \in \mathcal{C}(n;n)$ be a word of length $k$ in $\{\tau_i\}$; specifically
	\begin{equation}\label{eq:tau}
		\tau = \tau_{i_1}\tau_{i_2} \cdots \tau_{i_k}.
	\end{equation}
	Working with a 
braid representative of $\tau$, we let 
	$0 = t_0 <$ $t_1 <$ $t_2 <$ $\ldots$ $< t_{k-1}$ $< t_k = 1$
	be such that the restriction of the $t$ parameter to $[t_{j-1},t_j]$ parametrizes $\tau_{i_j}$ in \eqref{eq:tau}. Recall from \eqref{eq:A-set} the subsets $A_j := A_{t_j}$ and $S_n$, which are the $(n-1)$-torus with fixed last coordinate $t_j$ and the $n$th coordinate circle, respectively. We have the following cofibration diagram together with the induced exact sequence of sets and groups:
	\begin{diagram}
		\left(A_1 \sqcup A_2 \sqcup \ldots \sqcup A_k\right) \cup S_n & \rTo^{\subset} & \T^n & \rTo^{h} & \bigvee_k \Sigma \T^{n-1} \\
		\left[ \bigsqcup_{i=1}^k A_i ,\Cnf(n)\right] \times \pi_1\Cnf(n)  & \lTo[l>=.3in]^{\subset^\#} & [\T^n, \Cnf(n)] & \lTo[l>=.3in]^{h^\#} & \left[\bigvee_k \Sigma \T^{n-1}, \Cnf(n)\right] 
	\end{diagram}
		
	Since $\kappa(\widehat{\tau})$ restricted to any $A_j$ is null and $\Cnf(n)$ is simply-connected, there exists some ${z \in [\bigvee_k \Sigma\T^{n-1},\Cnf(n)]}$ such that $h^{\#}(z) = \kappa(\widehat{\tau})$. Let $r_j: \bigvee_k \Sigma\T^{n-1} \longrightarrow \Sigma\T^{n-1}$ be the projection onto the $j$th factor and let $z_j = r_j^{\#}(z)$. Clearly $z = z_1 \cdot \ldots \cdot z_k$ where $\cdot $ is the coproduct in $[\Sigma\T^{n-1},\Cnf(n)]$. Since we can choose the representative of $\tau_{i_j}$ to be $1 \cdot \ldots \cdot 1 \cdot \tau_{i_j} \cdot 1 \cdot \ldots \cdot 1$, repeating the above reasoning yields
	\[
		\kappa(\widehat{\tau}_{i_j}) = \kappa([\widehat{1 \cdot \ldots \cdot 1 \cdot \tau_{i_j} \cdot 1 \cdot \ldots \cdot 1}]) = h^{\#}(1 \cdot \ldots \cdot z_j \ldots \cdot 1) = h^{\#}(r_j^{\#}(z_j)) = p^{\#}(z_j),
	\]
	where the last identity follows from the definitions of $h$, $r_j$, and $p$. But then we know from \eqref{eq:CDstep1} that $\kappa(\widehat{\tau}_{i_j}) = p^\#(t_{n,i_j}(i_j))$, so it follows that the adjoint of $z_j$ is $t_{n,i_j}(i_j)$. Since the adjoint of the coproduct is the loop product, we conclude that the adjoint of $z$ equals $\phi(\tau)$, and 
	\[
		\kappa(\widehat{\tau}) = h^{\#}(z) = p^{\#}(\phi(\tau)).\qedhere
	\]
\end{proof}
For \ref{item:kappa-mu} of the Main Theorem, consider any $[L] \in BLM(n)$. Expanding in $\{B(n,\sigma)\}$ we obtain
\[
	\kappa([L]) = \sum_{\sigma \in \Sigma(2,\ldots , n-1)} e(1, \sigma(2), \ldots , \sigma(n-1)) B(n,\sigma),
\]
where $e(1,\sigma(2),\ldots, \sigma(n-1))$ are integer coefficients. We want to show that
\begin{equation}\label{eq:e=mu}
	e(1,\sigma(2),\ldots , \sigma(n-1)) = \mu(1,\sigma(2),\ldots , \sigma(n-1);n),
\end{equation}
where $\mu(1,\sigma(2),\ldots , \sigma(n-1);n)$ are the Milnor invariants of $L$. In the next two paragraphs we briefly recall Milnor's~\cite{Milnor57} definition of these invariants. 

For a fixed $n$-component link $L$ in $S^3$, the quotient $\mathcal{G}_q := \pi_1(S^3 -L)/(\pi_1(S^3-L)_q)$ is an isotopy invariant of $L$ and is generated by $n$ meridians $m_1, \ldots , m_n$, one for each component of $L$~\cite{Chen52,Milnor57}. Letting $w_1, \ldots , w_n$ be the words in $\mathcal{G}_q$ representing untwisted longitudes of the components of $L$, Milnor showed that $\mathcal{G}_q$ has the presentation
\[
	\langle m_1, \ldots , m_n | [m_1,w_1],\ldots , [m_n,w_n],R\rangle,
\]
where $R$ consists of all length $q$ commutators in the generators. If $F(m_1, \ldots , m_n)$ is the free group generated by the $m_i$ and $\Z\langle \langle X_1, \ldots , X_n\rangle \rangle$ is the ring of formal power series in the non-commuting variables $\{X_1, \ldots , X_n\}$, then the homomorphism $M: F(m_1, \ldots , m_n) \to \Z\langle \langle X_1, \ldots , X_n\rangle \rangle$ -- called the \emph{Magnus expansion} -- is defined on generators by
\[
	M(m_i) = 1+X_i, \qquad M(m_i^{-1}) = 1-X_i+X_i^2-X_i^3+\ldots .
\]

The Magnus expansion descends to $\mathcal{G}_q$ and the image of the longitude $w_j$ is
\begin{equation}\label{eq:M-mu}
	M(w_j) = 1+\sum_I \mu(I;j) X_I,	
\end{equation}
where the summation extends over all $I = (i_1, \ldots , i_m)$ where $1 \leq i_r \leq n$ and $X_I = X_{i_1}\cdots X_{i_m}$ for $m>0$. The coefficient $\mu(I;j)$ does not depend on $q$, if $q \geq m$. If $\Delta(I)$ is the greatest common divisor of the $\mu(J;j)$ with $J$ is a proper subset of $I$ (up to cyclic permutation), then Milnor showed that the residue class $\bar{\mu}(I;j)$ of $\mu(I;j)$ modulo $\Delta(I)$ is an invariant of the link $L$. If the indices of $I$ are all distinct, this is a link homotopy invariant; as alluded to above, $\Delta(I) = 0$ when $L$ is an $n$-component homotopy Brunnian link and $I$ is a length $n-1$ multiindex with no repeats, so there is no indeterminacy and $\mu(I;j)$ is a link homotopy  invariant of $L$.

As a first step to proving \eqref{eq:e=mu}, let $z \in \mathcal{BH}(n)$ be such that $\widehat{z} = L$. Thanks to the normal form \eqref{eq:normal-form-RF}, the definition of $\phi$, and Diagram~\ref{diag:big}, we have
\begin{equation}\label{eq:z}
	z = \prod_\sigma \tau(n, \sigma)^{e(1,\sigma(2), \ldots , \sigma(n-1))}.
\end{equation}
Therefore, it suffices to confirm \eqref{eq:e=mu} for a  link diagram obtained by closing up a pure braid representative of $z$ in the above normal form. The identity \eqref{eq:e=mu} will follow from 
\begin{equation}\label{eq:mu-kronecker}
	\begin{split}
		\mu(1, \xi(2), \ldots , \xi(n-1); n)(\widehat{\tau(n,\sigma)}) & = \begin{cases} 1 & \text{ if } \xi = \sigma \in \Sigma_{n-2} \\ 0 & \text{ if } \xi \neq \sigma, \end{cases} \\
		\mu(i_1,i_2, \ldots , i_k; n)(\widehat{\tau(n,\sigma)}) & = 0 \quad \text{for} \quad k < n-1,
	\end{split}
\end{equation}
and the product formula \cite[p. 7]{Mellor00}
\begin{equation}\label{eq:mu-prod}
	\begin{split}
		\mu(j_1,\ldots , j_p; n)(\widehat{z_1 \cdot z_2}) & = \mu(j_1, \ldots , j_p; n)(\widehat{z}_1) + \mu(j_1, \ldots , j_p;n)(\widehat{z}_2) \\
		&  \quad + \sum_{k=1}^{n-1} \mu(j_1, \ldots , j_k; n)(\widehat{z}_1) \mu(j_{k+1}, \ldots , j_p; n)(\widehat{z}_2),
	\end{split}
\end{equation}
where $z_1, z_2 \in \mathcal{H}(n)$ and  $\cdot$ denotes the product of string links. Indeed, any proper sublink $K$ of $\widehat{z} = L$ is link-homotopically trivial, meaning that any word $w_i$ representing a longitude in the link group ${\pi_1(S^3 - K)/\pi_1(S^3 - K)_q}$ is trivial. But then, since removing a component doesn't affect the $\mu$ invariants not involving the index of that component, this implies that all $\mu(j_1, \ldots , j_{k};n)(L)$ with $k  < n-1$ vanish. Hence, \eqref{eq:mu-prod} applied to \eqref{eq:z} yields
\[
	\begin{aligned}
		\mu(j_1,\ldots , j_{n-1};n)(\widehat{z}) & =  \sum_{\sigma} e(1, \sigma(2),\ldots , \sigma(n-1))\mu(j_1, \ldots , j_{n-1};n)(\widehat{\tau(n,\sigma)}).
	\end{aligned}
\]
Substituting $(j_1, \ldots , j_{n-1}) = (1, \sigma(2), \ldots , \sigma(n-1))$ and applying \eqref{eq:mu-kronecker}, we obtain \eqref{eq:e=mu} and \ref{item:kappa-mu} of the Main Theorem. One corollary is the fact, previously known to Milnor \cite{Milnor54}, that the $(n-2)!$ higher linking numbers $\{\mu(1, \sigma(2), \ldots , \sigma(n-1);n)\}$ separate $BLM(n)$.

It remains to prove \eqref{eq:mu-kronecker}. A longitude $w_n$ of the $n$th component of $\tau(n,\sigma)$ can be read off directly from the braid diagram of $\widehat{\tau(n,\sigma)}$ as the following word in the meridians $\{m_j\}$ of the other components:
\[
	m_\sigma = [m_1, m_{\sigma(2)}, \ldots , m_{\sigma(n-1)}].
\]
Formally, this word is simply obtained from $\tau(n,\sigma)$ by replacing $\tau_{n,j}$ with $m_j$. It is a general fact (see \cite{Magnus-Karrass-Solitar76}) that the leading term in the Magnus expansion of any commutator $[m_i,m_j]$ is equal to the commutator $[X_i,X_j] = X_iX_j - X_jX_i$ in the ring $\Z\langle\langle X_1, \ldots , X_n \rangle\rangle$ provided $[X_i,X_j] \neq 0$. Therefore, an inductive argument implies that
\[
	M(w_n) = 1 + [X_1,X_{\sigma(2)}, \ldots , X_{\sigma(n-1)}] + (\text{higher-order terms})
\]
(notice that $[X_1, X_{\sigma(2)}, \ldots , X_{\sigma(n-1)}] \neq 0$ since it involves distinct variables). Again, an elementary induction on $n$ shows that in the expansion of $X_\sigma := [X_1$, $X_{\sigma(2)}$, $\ldots$ , $X_{\sigma(n-1)}]$ in monomials of degree $n-1$, the monomial $X_1X_{\sigma(2)}\cdots X_{\sigma(n-1)}$ occurs only once and is a leading term in the expansion.
Therefore, \eqref{eq:mu-kronecker} follows from \eqref{eq:M-mu}. This completes the proof of the Main Theorem.

Corollary \ref{cor:homotopy-periods} is a direct consequence of the fact that the homotopy periods of the Samelson products $\{B(n,\sigma)\}$ in $\pi_{n-1}(\Omega\Cnf(n))$ can be obtained from
the general methodology of Sullivan's minimal model theory \cite{Sullivan77}, or approaches in \cite{Hain84}, \cite{Sinha-Walter08}. Consult \cite{Kom-higher-helicity2} for a basic derivation of such an integral in the case of three component links.


\providecommand{\bysame}{\leavevmode\hbox to3em{\hrulefill}\thinspace}
\providecommand{\MR}{\relax\ifhmode\unskip\space\fi MR }
\providecommand{\MRhref}[2]{%
  \href{http://www.ams.org/mathscinet-getitem?mr=#1}{#2}
}
\providecommand{\href}[2]{#2}
\bibliography{brunnian}
\bibliographystyle{amsplain}

\end{document}